\def\seq#1_#2{\langle #1_#2:#2\in\omega\rangle}
\def\fc#1|#2{#1\uparrow#2}
\def\ain{\subseteq^*}

\def\C{{\mathfrak c}}

\def\set#1:#2.{{\{\,#1:#2\,\}}}

\def\so{\mathop{\rm so}}
\def\Fn{\mathop{\rm Fn}}
\def\supp{\mathop{\rm supp}}
\let\namef\dot
\def\Hc{{H(2^{\C^+})}}
\def\force#1``#2''{\Vdash#1\hbox{``}#2\hbox{''}}
\def\K{{\cal K}}
\def\Ka{{\cal K}_\alpha}
\def\Kb{{\cal K}_\beta}

\def\Ua{{\cal U}_\alpha}
\def\Ub{{\cal U}_\beta}

\def\MMd{{\hbox{$\psi^*$}-derived}}
\def\G{{\mathbb G}}

\documentclass{article}
\title{Convergence in topological groups and the Cohen reals}
\author{Alexander Y.~Shibakov\footnote{Tennessee Tech.\ University, email: {\tt ashibakov@tntech.edu}}}
\usepackage{amssymb}
\usepackage{amsthm}
\usepackage{amsmath}
\usepackage{showlabels}
\usepackage{enumitem}
\newtheorem{theorem}{Theorem}
\newtheorem{lemma}{Lemma}
\newtheorem{example}{Example}

\newtheorem{definition}{Definition}
\newtheorem{question}{Question}
\begin{document}
\maketitle
\begin{abstract}
We show that after adding $\omega_2$ Cohen reals to a model of
$\diamondsuit$ the resulting extension has no countable sequential
groups of intermediate sequential order. On the other hand there
exists an uncountable sequential group of sequential order $2$ in the
same extension.
\end{abstract}
\setlist[enumerate,1]{label={\rm(\arabic*)},ref={(\arabic*)}}
\setlist[enumerate,2]{label={\rm(\alph*)},ref={(\arabic{enumi}.\alph*)}}
\section{Introduction}
The study of convergence properties in the presence of a group
structure has shown that the restrictions placed by the algebra
on the topology of the group depend greatly on the set
theoretic assumptions. In their celebrated paper~\cite{HRG} M.~Hru\v
s\'ak and Ramos-Garc\'ia solve a long standing problem due to Malykhin
by constructing a model of ZFC in which every countable Fr\'echet
group is metrizable. In~\cite{S1} it is shown that there exists a
model of ZFC in which every sequential group (whether countable or
uncountable) of intermediate sequential order is Fr\'echet (see below
for the exact definitions), answering a question of P.~Nyikos. On the
other hand, CH implies the existence of many pathological sequential
groups (see~\cite{S2}) including nonmetrizable countable Fr\'echet
groups.

The authors of \cite{HRG} ask how strong the interaction between
algebra and convergence really is, in particular, whether the
existence of a nonmetrizable Fr\'echet group topology on one countable
group implies that every abelian group has such a topology. This question is still
open.

In this paper we show that after adding $\omega_2$
Cohen reals to a model of $\diamondsuit$ the resulting extension has
no countable sequential groups of intermediate sequential order, while
there is an uncountable such group (see Theorem~\ref{sggap}). This
shows that both the existence and the size of the group of an
intermediate sequential order depend on the set theoretic
assumptions. This is in contrast to the situation with the F\'echet
groups where an uncountable nonmetrizable Fr\'echet group exists in ZFC.
\section{Definitions and terminology}
Our general set theoretic terminology is standard and
follows~\cite{Ku}. $\Fn(I,2)$ will stand for the set of functions
$p:J\to2$ where $I\supseteq J=\supp p$ is finite, with the natural order.
All spaces are assumed to be Hausdorff unless stated otherwise.

If $X$ is a topological space and $C\subseteq X$, $x\in
X$, we write $C\to x$ to indicate that $C$ converges to $x$, i.e.\
$C\ain U$ for every open $U\ni x$.
Recall that a space $X$ is called {\it sequential\/} if for every $A\subseteq X$
such that $\overline{A}\neq A$ there is a $C\subseteq A$ such that
$C\to x\not\in A$.

Let $A\subseteq X$. Define the {\it sequential closure of $A$}, $[A]'
= \{\,x\in X: C\to x\text{ for some }C\subseteq A\,\}$.
Put $[A]_\alpha = \cup\{[A_\beta]':\beta<\alpha\}$ for $\alpha\leq\omega_1$.
Define $\so(X) = \min\{\,\alpha \leq \omega_1 : [A]_\alpha =
\overline{A}\text{ for every }A\subseteq X\,\}$. $\so(X)$ is called
the {\it sequential order\/} of $X$. Spaces of sequential order $\leq
1$ are called Fr\'echet.

Let $X$ be an arbitrary set and ${\cal
K}\subseteq2^X$, so that each $K\in{\cal K}$ is equipped with a
topology $\tau_K$.
Then one can introduce a topology on $X$ by making a
$U\subseteq X$ open in $\tau$ if and only if each $U\cap K$
is open in $\tau_K$. We will say that $\tau$ is {\it determined\/} by
${\cal K}$. While as a subspace of $X$ in this topology, the
topology of $K\in{\cal K}$ may be different from $\tau_K$,
if each $(K,\tau_K)$ is compact and Hausdorff and for any
$K,F\in{\cal K}$ the intersection $K\cap F$ is closed in each $K$ and
$F$ and inherits the same topology from either space then the topology of each
$K\in{\cal K}$ as a subspace of $X$ is exactly $\tau_K$. Note that
$\tau$ is $T_1$ if and only if each $\tau_K$ is.

If the family ${\cal K}$ above is countable, consists of compact
spaces and satisfies the condition mentioned at the end of the
previous paragraph, $X$ will be called a {\it $k_\omega$ space}.

We use~\cite{AT} as a general reference on topological groups. In
abelian groups $0$, $+$, and $-$ stand for the unit, the group
operation and the algebraic inverse, respectively.
It will be convenient to use $\sum^kB$ for the
sum $B+\cdots+B$ of $k$ copies of some subset $B$ of an abelian group
$G$ where $\sum^0B=\{0\}$. A collection ${\cal U}$ of subsets of $G$ is called
{\it translation invariant\/} if for any $g\in G$ and any $U\in{\cal
U}$ $g+U\in{\cal U}$.

The smallest cardinality of a family ${\cal U}$
of open neighborhoods of $x$ such that $\cap{\cal U}=\{x\}$ is called
the {\it psudocharacter $\psi(x,X)$ of $X$ at $x\in X$\/}. When $G$
is a topological group we write simply $\psi(G)$.
An uncountable group $G$ will be called {\it co-countable\/} if every
quotient of $G$ that has countable pseudocharacter is countable.

Recall that a (necessarily abelian) group $G$ is called {\it boolean\/} if $a+a=0$
for every $a\in G$.

The following lemma is likely folklore, see~\cite{S3} for a proof.
\begin{lemma}\label{komega}
Let $G$ be a group, and $\K$ be a cover of $G$.
Suppose each $K\in\K$ is given a compact Hausdorff topology such that for
any $K',K''\in\K$ the set $K'\cap K''$ is closed in both $K'$ and
$K'$ and the induced topologies are the same. Suppose further that
the sums, inverses, and unions of any finite number of $K$'s are contained in
some $K'$'s and that
the addition and algebraic inverse maps restricted to
the corresponding (finite products of) compacts are continuous (for any large
enough compact range). Then the
topology $\tau$ determined by $\K$ on $G$ is translation invariant. If
$\K$ is countable $\tau$ is a Hausdorff group topology.
\end{lemma}
The definition below can be viewed as an approximation of a group
topology on $G$.
\begin{definition}
Let $G$ be an abelian topological group, ${\cal U}$ be a family of open subsets
of $G$. Call ${\cal U}$ {\it centered\/} if it is translation
invariant, closed under finite intersections, $H=\cap\set U:0\in
U\in{\cal U}.$ is a closed subgroup of $G$,
$U+H=U$ for every $U\in{\cal U}$, and $\set p(U):U\in{\cal U}.$ is a
base of some Hausdorff topology on $G/H$ where $p:G\to G/H$ is the
natural quotient map.
\end{definition}

The next simple lemma provides a basic method of constructing centered families.
\begin{lemma}\label{opensat}
Let $G$ be a co-countable abelian group. If ${\cal U}$ is a countable centered
family of open subsets of $G$ and $g\in U\subseteq G$ is open then ${\cal U}$ can be extended to a
countable centered family ${\cal U}'$ of open subsets of
$G$ such that for some $U'\in{\cal U}'$ $g\in U'\subseteq U$.
\end{lemma}
\begin{proof}
Let $H=\cap\set V:0\in V\in{\cal U}.$. By shifting $U$ if necessary,
assume $g=0$. Extend the family ${\cal
U}\cup\{U\}$ if necessary to a family ${\cal U}'''$ of open subsets of
$G$ to ensure that the new family is countable,
closed under finite intersections, and for any $0\in V\in{\cal U}'''$ there
are $O,W\in{\cal U}'''$ such that $0\in O\cap W$ and $O-W\subseteq
V$. Put $H'=\cap\set V:0\in V\in{\cal U}'''.$ and ${\cal U}''=\set
V+H':V\in{\cal U}'''.$. Then $H'\subseteq H$ is a
closed subgoup of $G$ so ${\cal U}\subseteq{\cal U}''$. The group
$G/H'$ is countable so ${\cal U}''$ can be extended to a centered
family ${\cal U}'$ by extending the family $\set p(V):V\in{\cal U}''.$
to a translation invariant base of open subsets of a Hausdorff
topology on $G/H'$ closed under finite intersections where $p:G\to
G/H'$ is the natural quotient map. To see that $U'\subseteq U$ for
some $U'\in{\cal U}''$ let $V,W\in{\cal
U}'''$ be such that $0\in V\cap W$ and $V-W\subseteq U$. Then
$U'=V+H'\subseteq U$.
\end{proof}

Recall that $A,B\subseteq\omega$ are said to be {\it almost
disjoint\/} if $|A\cap B|<\omega$. An infinite family ${\cal
A}\subseteq[\omega]^\omega$ is called a {\it maximal almost
disjoint\/} or {\it MAD\/} family if any two different sets in ${\cal
A}$ are almost disjoint and every infinite subset of $\omega$ that is
almost disjoint with every element of ${\cal A}$ is in ${\cal A}$.

The following result is a corollary of \cite{Ku}, Theorem~VIII.2.3.

\begin{lemma}[\cite{Ku}]\label{kunen}
If $V\vdash\hbox{CH}$ there exists a MAD family ${\cal A}\in V$ such
that ${\cal A}$ remains a MAD family in $V[\G]$ where $\G$ is
$\Fn(\omega_2,2)$-generic.
\end{lemma}

As shown in~\cite{H}, such a family also exists in any model
satisfying ${\mathfrak b}=\C$. The MAD family of Lemma~\ref{kunen} is used
in the example below to construct a compact sequential space that
retains its convergence properties after adding Cohen reals
(see~Lemma~\ref{pstable} below).

\begin{example}\label{psi}Let ${\cal A}\subseteq[\omega]^\omega$ be a
maximal almost disjoint family from Lemma~\ref{kunen}. Put
$\psi^*=\{\omega\}\cup{\cal A}\cup\omega$ and define a topology on
$\psi^*$ by making each $n\in\omega$ isolated, making $\set A\setminus
n \cup\{A\}:n\in\omega.$ the base of open 
neighborhoods at $A\in{\cal A}$, and making
$\set\psi^*\setminus((\cup{\cal A}')\cup{\cal A}'):{\cal
A}'\subseteq{\cal A}\hbox{ is finite}.$ the base of open neighborhoods 
at $\omega$. Then $\psi^*$ is a compact sequential space such that
$\so((\psi^*)^n)=2$ for every $n\in\omega$ (see~\cite{Ka}).
\end{example}

Let $\psi^*$ be the compact space from Example~\ref{psi}. Call a compact
$K$ {\em\MMd\/} if $K$ is a continuous image of a closed subspace of
$(\psi^*)^n\times C^m$ for some $m,n\in\omega$ where $C$ is a
convergent sequence. Note that $\so(K)\leq 2$ for each \MMd\ compact
$K$ since perfect maps do not raise the sequential order (see~\cite{Ka}).

\begin{lemma}\label{pstable}
Let $V'$ be an extension of $V$ in which the MAD family from
Example~\ref{psi} remains a MAD family. Let $K\in V$ be an \MMd\ compact
space with the topology whose base is the topology of $K$ in
$V$. Then $K$ is a sequential compact space in $V'$ such that $\so(K)\leq2$.
\end{lemma}
The following lemma follows from the proof of Lemma~5
of~\cite{S3}. We show the proof that $U\in{\cal U}$ remain open only.
\begin{lemma}\label{sextend}
let $G$ be a boolean $k_\omega$ group whose topology is determined by a
countable family $\K=\set K_n:n\in\omega.$ of compact subspaces closed
under the usual operations. Let ${\cal U}$ be a centered family of open subsets
of $G$, let $H=\cap\set U:0\in U\in{\cal U}.$ and $p:G\to G/H$ be the natural quotient map. Let $C=\seq
c_n\subseteq G$, $G'\subseteq G$ be the subgroup generated by $C$
which satisfy the following properties: 
\begin{enumerate}
\item\label{s.converge} $C\ain U$ for every $U\in{\cal U}$ such that $0\in U$

\item\label{s.lifting} for any $c\in G'$ and any $K\in\K$ the intersection $(c+H)\cap
K\not=\varnothing$ if and only if $c\in K$

\item $c_n\not\in(\cup_{i<n}K_i+C_n)-(\cup_{i<n}K_i+C_n)$ where
$C_n=\set c_i:i<n.$
\end{enumerate}
Let $S=C\cup\{0\}$ with the unique topology such that $C\to0$. Let
$\K^+$ be the closure of $\K\cup\{S\}$ under the usual
operations. Then the topology determined by $\K^+$ is a Hausdorff
group topology on $G$ in which each $U\in{\cal U}$ is open.
\end{lemma}

\begin{proof}
Since by~\ref{s.converge} $p(C)\to0$ in the topology generated by the sets $p(U)$,
$U\in{\cal U}$ it is enough to show that the map $p$ is continuous in
the topology determined by $\K^+$. Thus one needs to show that $H$ is
closed in the topology determined by $\K^+$. Every $K^+\in\K^+$ is a
subspace of some $K+\sum^kS$ where $K\in\K$. Let $T=\set
k_n+s_n:n\in\omega.\subseteq (K+\sum^kS)\cap H$ be such that $T\to g$
for some $g\in G$ in the topology determined by $\K^+$ and each
$k_n\in K$, $s_n\in\sum^kS$. Then $s_n+H\ni k_n$ and
by~\ref{s.lifting} $s_n\in K$. Thus $T\subseteq K+K$ so $g\in H$.
\end{proof}
The next lemma is a corollary of Lemma~16 and Corollary~2 in \cite{S1}.
\begin{lemma}[\cite{S1}]\label{gs}
Let $G$ be a countable sequential group such that
$1<\so(G)<\omega_1$. Let $\seq N_n$ be a family of nowhere dense
subsets of $G$. Then there exists a $C\subseteq G$ such that $C\to0$
and $C\cap N_n$ is finite for every $n\in\omega$.
\end{lemma}

\section{Countable groups in the Cohen model}
The proof of the following lemma is based on the ideas of~\cite{DB}, Theorem~3.1.
\begin{lemma}\label{ncsgisq}
Let $V\vdash\hbox{CH}$ and $\G$ be
$\Fn(\omega_2,2)$-generic over $V$. Then in $V[\G]$ there are no
countable sequential groups $G$ such that $1<\so(G)<\omega_1$.
\end{lemma}

\begin{proof}
Let $\namef\tau$ be an $\Fn(\omega_2,2)$-name such that
$$
q\force_{\Fn(\omega_2,2)}``\namef\tau\hbox{ is a sequential group
topology on }\omega\hbox{ such that }1<\so(\namef\tau)<\omega_1''
$$
for some $q\in\G$ and let $M$ be an $\omega$-closed (i.e.\
$[M]^\omega\subseteq M$) elementary submodel of $\Hc$ of size
$\omega_1$ such that $q,\namef\tau\in M$. We will treat $\namef\tau$ as
an $\omega_2$ long listing of countable names for subsets of $\omega$
so that $\namef\tau(\alpha)$ will refer to the name of the element of
$\namef\tau$ with index $\alpha<\omega_2$. We will also assume that
countable names of all the relevant countable structures (such as the
group operation) needed later are in $M$.

Let $P=\Fn(\omega_2\cap M,2)$, and let $\tau_M$ consist of
$\G\cap M$-interpretations of $\namef\tau(\alpha)$ such that $\alpha\in
M$. Then $\tau_M\in V[\G\cap M]$ and standard arguments show that
$\tau_M$ is a base of a group topology on $\omega$ (note that since
the name of the group operation on $\omega$ is countable and in $M$
the operation itself is in $V[\G\cap M]$ and that $\omega$-closedness
of $M$ implies that $\tau_M$ is a topology). Let $A\in V[\G\cap M]$ be a
subset of $\omega$ such that $A$ is not closed in $\tau_M$. Let
$\namef A$ be a $P$-name of $A$. Using $\omega$-closedness of
$M$ we may assume that $\namef A\in M$. Then there exists an
$m\in\omega$ such that $M\vdash q'\force``\forall \alpha<\omega_2:\namef\tau(\alpha)\cap\namef
A\not=\varnothing\hbox{ if }m\in\namef\tau(\alpha)''$ for some $q'\leq
q$, $q'\in\G$. Then
$$
q'\force_P``\namef S\subseteq\namef A\hbox{ and }\namef
S\to \namef n\in\omega\setminus\namef A''
$$
for some $\namef S,\namef n\in M$. Thus the $\G\cap
M$-interpretation $S$ of $\namef S$ is in $V[\G\cap M]$ and is a
convergent sequence in $\tau_M$ such that $S\subseteq A$ and $S\to
n\in\omega\setminus A$. Therefore $\tau_M$ is sequential.

For every $\alpha<\omega_1$ the statement $\so(\tau_M)>\alpha$ is
witnessed by some countable set $S\in V[\G\cap M]$ with a countable
name $\namef S\in M$. Using
$q\force_{\Fn(\omega_2,2)}``1<\so(\namef\tau)<\omega_1''$ and the
elementarity of $M$ one can conclude that $1<\so(\tau_M)<\omega_1$ in
$V[\G\cap M]$.

Let $U\in V[\G]=V[\G\cap M][\G\setminus M]$ be a subset of $\omega$ open
in $\tau$.
Using standard ccc arguments, pick an $\Fn(L,2)$-name
$\namef U\in V[\G\cap M]$ for $U$ where $L\subseteq\omega_2\setminus M$ is countable. Let $\seq
p_n$ list all the elements of $\Fn(L,2)$ and put $U_n=\set
i\in\omega:p_n\force_{\Fn(\omega_2\setminus M,2)}``i\in\namef U''.\in V[\G\cap M]$.

If $\namef S$ is a countable $P$-name of a
sequence $S\subseteq\omega$ that converges to some $m\in\omega$ in $\tau_M$ and such that
$S\in V[\G\cap M]$ then one may assume that $\namef S\in M$ by
$\omega$-closedness.  Pick a $q'''\leq
q$, $q'''\in\G$ such that $q'''\force_{\Fn(\omega_2,2)}``\namef
S\ain\namef\tau(\alpha)\hbox{ for any }\alpha\in M\hbox{ such that
}m\in\namef\tau(\alpha)''$ and put $q''=q'''\cap M$. If $S\not\to m$ in
$\tau$ there exists a $q'\leq q''$ such that $q'\force_{\Fn(\omega_2,2)}``\namef 
S\not\ain\namef\tau(\alpha)\hbox{ and }m\in\namef\tau(\alpha)\hbox{ for
some }\alpha<\omega_2''$. By elementarity one can assume that
$q'\in M$. Then $q'''$ and $q'$ are compatible, a contradiction.

Hence every sequence that converges in $\tau_M$ is
also convergent in $\tau$ and the (sequential) closure of $U_n$ in $\tau_M$ is
contained in the (sequential) closure of $U_n$ in $\tau$.

Suppose $r'\in\G$, $r'\force_{\Fn(\omega_2,2)}``0\in\namef U''$,
$r=r'\cap \Fn(L,2)$, and each $U_n$ such that $p_n\leq r$ is nowhere dense in $\tau_M$. Then
by Lemma~\ref{gs} and
$1<\so(\tau_M)<\omega_1$ there exists a
sequence $S\in V[\G\cap M]$ such that $S\to0$ in $\tau_M$ and the set
$S\cap U_n$ is finite for every $n\in\omega$ such that $p_n\leq r$. By
extending $r'$ if necessary we may assume that
$r'\force_{\Fn(\omega_2,2)}``\namef S\to0\hbox{ in }\namef\tau''$ for
some name $\namef S\in M$ of $S$.

Since $U$ is open in $\tau$ there exist a $k\in\omega$ and a
$p\in\Fn(\omega_2\setminus M,2)$ such that $p\leq
r'\cap(\omega_2\setminus M)$ and $p\force_{\Fn(\omega_2\setminus
M,2)}``S\setminus k\subseteq\namef U''$. Let $p_n=p\cap\Fn(L,2)$. Then
$p_n\leq r$ and $S\setminus k\subseteq U_n$. Otherwise there would exist an $m\in
S\setminus k$ and a $p'\leq p$ such that
$p'\force_{\Fn(\omega_2\setminus M,2)}``m\not\in \namef U''$ contrary
to the choice of $p$. The property of $p_n$ thus contradicts the
choice of $U_n$ and $S$.

Thus for every $r'\in\G$ the set $\overline{U_n}$ has a nonempty interior in $\tau_M$ for some
$n\in\omega$ such that $p_n\leq r'\cap\Fn(L,2)$. So $p_n\in\G$,
$\overline{U_n}$ has a nonempty interior in $\tau_M$, and $U_n\subseteq U$ for some
$n\in\omega$. Therefore, for every open $U$ in $\tau$ its closure
$\overline{U}$ contains a $V\in\tau_M$. This implies that $\tau$ has a
$\pi$-base of size $\omega_1$. Since $\tau$ is a group topology its
weight is at most $\omega_1$. If $\tau$ is not Fr\'echet, the group contains
a subspace homeomorphic to $S(\omega)$ (see \cite{Ny}). But $S(\omega)$
has a character of $\omega_2$ in $V[\G]$, a contradiction.
\end{proof}

\section{An uncountable group}
We assume below that $V\vdash\diamondsuit$.

Let $B\subseteq2^{\omega_1}$ be a subspace of $2^{\omega_1}$
homeomorphic to the space $\psi^*$ from Example~\ref{psi}. View
$2^{\omega_1}$ as a boolean group and arrange for the only point of
$B$ of Cantor-Bendixon rank 2 to be $0\in2^{\omega_1}$. Let $G$ be the
group algebraically generated by $B$. Let
$\K_0=\set\sum^nB:n\in\omega.$ and $\tau_0$ be the topology determined
by $\K_0$. Then $\tau_0$ is a co-countable Hausdorff group topology by
Lemma~\ref{komega}. Let ${\cal U}_0=\{G\}$.

Let $F:G\times\Fn(\omega_1,2)\to\omega_1$ be a one to one
map. If $\namef A$ is a name of a subset of $G$ one can think of
$\namef A$ as a subset of $G\times\Fn(\omega_1,2)$, or using
$F(\namef A)$ instead, as a subset of $\omega_1$. Conversely, every
subset of $\omega_1$ can be naturally interpreted as a name of a
subset of $G$. Let $\set C_\alpha:\alpha\in\omega_1.$ be a
$\diamondsuit$-sequence where each $C_\alpha$ is interpreted as an
$\Fn(L_\alpha,2)$-name as above for some countable
$L_\alpha\subseteq\omega_1$. We will write $\pi_1(C_\alpha)=\set g\in
G:(g,p)\in F^{-1}(C_\alpha)\hbox{ for some }p\in\Fn(L_\alpha,2).$. Let
$G=\set g_\alpha:\alpha<\omega_1\hbox{ is limit}.$

Construct decreasing co-countable topologies $\tau_\alpha$, increasing countable
families of compact subspaces $\Ka$, and increasing countable families
$\Ua$ of subsets of $G$ so that

\begin{enumerate}[series=conditions]
\item\label{subCa} if $\alpha$ is a limit ordinal and $C_\alpha$
has the following properties:
\begin{enumerate}
\item\label{subCa.dense} $\force_{\Fn(\omega_1,2)}``C_\alpha\cap
U\not=\varnothing\hbox{ for every }U\in\Ua\hbox{ such that }0\in U''$

\item\label{subCa.free} for every $K\in\Ka$ there is a $\namef U_K$
such that $\force_{\Fn(\omega_1,2)}``0\in\namef U_K\in\Ua\hbox{ and }C_\alpha\cap (K\cap
\namef U_K)=\varnothing''$

\item\label{subCa.indep} the group $G_\alpha$ generated by
$\pi_1(C_\alpha)$ is such that $G_\alpha\cap H_\alpha=\{0\}$ and for
every $g\in G_\alpha$ and $K\in\Ka$
$(g+H_\alpha)\cap K\not=\varnothing$ if and only if $g\in K$
\end{enumerate}
then $S_{\alpha+1}\to0$ in $\tau_{\alpha+1}$ and
$\force_{\Fn(\omega_1,2)}``S_{\alpha+1}\cap C_\alpha\hbox{ is infinite}''$.

\item\label{G.sep} if $\alpha$ is a limit ordinal there exists a
$U\in{\cal U}_{\alpha+1}$ such that $g_\alpha\not\in U$ 

\item\label{G.ind} each $U\in\Ua$ is open in $\tau_\alpha$ where
$\tau_\alpha$ is a $k_\omega$ Hausdorff group topology determined by $\Ka$

\item\label{G.collpse} $\Ua$ is a centered family and
$H_\alpha=\cap\set U\in\Ua:0\in U.$ is a closed subroup of $G$ in $\tau_\alpha$

\item\label{G.derived} each $K\in\Ka$ is an \MMd\ compact, $\Ka$ is
closed under finite sums, unions, and intersections, and covers $G$
\end{enumerate}

Suppose $\alpha$ is a limit ordinal. Let $\Ua=\cup_{\beta<\alpha}\Ub$ and
$\Ka=\cup_{\beta<\alpha}\Kb=\seq K_n$. Let $0\not=g_\alpha\in G$. The
topology $\tau_\alpha$ determined by $\Ka$ is a Hausdorff group
topology on $G$ in which every $U\in\Ua$ is open
by~\ref{G.ind}, \ref{G.derived}, and Lemma~\ref{komega}.
Then~\ref{G.ind}--\ref{G.derived} follow from the construction.

Find a $U'\subseteq G$ such that $0\in
U'\not\ni g_\alpha$ and $U$ is open in $\tau_\alpha$. Extend the
family $\Ua$ if necessary to a countable centered family ${\cal U}_{\alpha+1}$ of
open in $\tau_\alpha$ subsets of $G$ such that~\ref{G.collpse} holds
and $0\in U\subseteq U'$ for some $U\in{\cal U}_{\alpha+1}$ using
Lemma~\ref{opensat}. Put $\set U:0\in U\in{\cal U}_{\alpha+1}.=\seq
U_n$.

Let $\seq p_n$ list all the elements of $\Fn(L_\alpha,2)$ such that
each is listed infinitely many times. Pick $s_n$ inductively for each
$n\in\omega$ so that
\begin{enumerate}[conditions]
\item\label{Cafree} $s_n\not\in(\cup_{i<n}K_i+S_n)-(\cup_{i<n}K_i+S_n)$ where
$S_n=\set c_i:i<n.$

\item\label{Caopen} $s_n\in\cap_{i=0}^n U_i$

\item\label{Cadense} there exists a $p_n'\leq p_n$ such that $p_n'\force_{\Fn(\omega_1,2)}``s_n\in
C_\alpha''$.

\end{enumerate}
The sum in \ref{Cafree} is a subset of some $K\in\Ka$ so
by~\ref{subCa.free} one can find a $p\leq p_n$ and a $U_m$ such
that $p\force_{\Fn(\omega_1,2)}``C_\alpha\cap(K\cap
U_m)=\varnothing''$. Applying~\ref{subCa.dense} find a $p_n'\leq p$
and an $s_n\in\cap_{i=0}^{\max\{n,m\}}U_i$ such that
$p_n'\force_{\Fn(\omega_1,2)}``s_n\in C_\alpha''$. By the choice of
$p$ $s_n\not\in K$. Then \ref{Cafree}--\ref{Cadense} follow.

Put $S_{\alpha+1}=\set s_n:n\in\omega.$ and 
let $\Ka'=\Ka\cup\{S_{\alpha+1}\cup\{0\}\}$. Let $\K_{\alpha+1}$ be the
closure of $\Ka'$ under finite sums, unions and intersections and $\tau_{\alpha+1}$ be the
$k_\omega$ group topology generated by
$\K_{\alpha+1}$. Lemma~\ref{sextend}, \ref{Caopen}, \ref{subCa.indep},
and~\ref{Cafree} imply that $\tau_{\alpha+1}$ is Hausdorff and each
$U\in{\cal U}_{\alpha+1}$ is open in
$\tau_{\alpha+1}$. Thus~\ref{G.ind}--\ref{G.derived} hold for $\alpha+1$.

Suppose there exist an $m\in\omega$ and a $p\in\Fn(\omega_1,2)$ such
that $p\force``s_n\not\in C_\alpha\hbox{ for }n\geq m''$. We may
assume that $p\in\Fn(L_\alpha,2)$ so $p=p_k$ for some
$k\in\omega$ such that $k>m$. Then by~\ref{Cadense} there exists a $p_k'\leq p_k$ such
that $p_k'\force``s_k\in C_\alpha''$, contradiction. Thus
$\force_{\Fn(\omega_1,2)}``S_{\alpha+1}\cap C_\alpha\hbox{ is
infinite}''$. So~\ref{subCa} holds.

If $\alpha$ is a successor the case of $\alpha=\beta+1$ for a limit
$\beta$ is discussed above. In other cases, trivial choices suffice.

Put $\K=\cup_{\alpha<\omega_1}\Ka$ and ${\cal
U}=\cup_{\alpha<\omega_1}\Ua$. In $V[\G]$ let the topology on each
$K\in\K$ be the topology whose base of neighborhoods consists of the
subsets open in the natural topology of $K$ in $V$. Let $\tau$ be the
topology on $G$ determined by $\K$. Then $\tau$ is translation
invariant by Lemma~\ref{komega} and sequential by Lemma~\ref{pstable}. Each
$U\in{\cal U}$ is open in $\tau$ by~\ref{G.ind}. The intersection
$\cap\set U\in{\cal U}:0\in U.=\{0\}$ by~\ref{G.sep} so ${\cal U}$ is
a translation invariant base of open set of a Hausdorff group topology
$\tau'$ on $G$ by~\ref{G.ind} and~\cite{AT}, Theorem~1.3.12. The
topology $\tau'$ is coarser than $\tau$ so sets $U\cap K$ form a base
of open sets for every $K\in\K$ where $U\in{\cal U}$.

The next lemma shows that $\tau$ and $\tau'$ coinside in $V[\G]$.
\begin{lemma}\label{groupsq}
There is no $\Fn(\omega_2,2)$-name $\namef A$ of a subset of $G$ such that
\begin{enumerate}[conditions]
\item\label{A2.dense'} $p\force_{\Fn(\omega_2,2)}``\namef A\cap U\not=\varnothing\hbox{
for every }U\in{\cal U}\hbox{ such that }0\in U''$

\item\label{A2.free'} $p\force_{\Fn(\omega_2,2)}``\forall
K\in\K\exists V:0\in V\subseteq K\hbox{ is relatively
open in $K$ and }\namef A\cap V=\varnothing\hbox{ if }0\in
K''$
\end{enumerate}
for some $p\in\Fn(\omega_2,2)$.
\end{lemma}
\begin{proof}
%for every $K\in\K$ such that $0\in K$ there exists a $\namef V_K$
%such that
Using standard abuse of terminology and possibly changing $\namef A$
if necessary treat $\namef A$ as a set of pairs
$(r,g)$ where $r\leq p$, $r\in\Fn(\omega_2,2)$, and $g\in G$.
Let $Z=\supp p$ and let $\namef A'$ be the name constructed by
replacing every pair $(q,g)\in \namef A$ by $(q|_{\omega_2\setminus
Z},g)$. Suppose there exists a $p'\in\Fn(\omega_2\setminus Z,2)$ and a $K\in\K$
such that
$$
p'\force_{\Fn(\omega_2\setminus Z,2)}``\forall V:\namef A'\cap
V\not=\varnothing\hbox{ if }0\in V\hbox{ and }V\hbox{ is
relatively open in }K''
$$
Using~\ref{A2.free'} and the property of $\tau'$ stated in the remark
before the lemma find a $U_K\in{\cal U}$ such that $0\in U_K$, and a $q'\in\Fn(\omega_2,2)$ such that
$q'|_{\omega_2\setminus Z}\leq p'$, $q'\leq
p$, and
$
q'\force_{\Fn(\omega_2,2)}``\namef A\cap (U_K\cap K)=\varnothing''
$.
Now
$
q'|_{\omega_2\setminus Z}\force_{\Fn(\omega_2\setminus Z,2)}``\namef A'\cap(U_K\cap K)=\varnothing''
$.
Otherwise there are $q''\in\Fn(\omega_2\setminus Z,2)$ and $g\in G$
such that $q''\leq q'|_{\omega_2\setminus Z}$ and
$
q''\force_{\Fn(\omega_2\setminus Z,2)}``g\in\namef A'\cap(U_K\cap K)''
$.
It follows from the construction of $\namef A'$ that
$
q''\force_{\Fn(\omega_2,2)}``g\in\namef A\cap(U_K\cap K)''
$
contradicting the choice of $q'$.

Now $q'|_{\omega_2\setminus Z}$ and $p'$ are compatible, a
contradiction. Therefore for every $K\in\K$
$$
\force_{\Fn(\omega_2\setminus Z,2)}``\exists V\subseteq K\hbox{
relatively open in $K$}:\namef A'\cap V=\varnothing\hbox{ and }0\in
V\hbox{ if }0\in K''
$$
Mapping $\omega_2\setminus Z$ onto $\omega_2$ in a one to one manner,
and changing $\namef A'$ appropriately, after using standard
arguments~\ref{A2.free'} can be replaced with
\begin{enumerate}[conditions]
\item\label{A2.free} for every $K\in\K$ there exists a $\namef U_K$ such that
$\force_{\Fn(\omega_2,2)}``0\in\namef U_K\in{\cal U}\hbox{ and }\namef A\cap(\namef U_K\cap K)=\varnothing''$
\end{enumerate}
A similar argument shows that $p$ can be omitted in~\ref{A2.dense'}.
Since $G$, ${\cal U}$, and $\K$ are of size $\omega_1$ standard arguments imply that $\namef
A$ and $\namef U_K$ may be chosen to be $\Fn(L,2)$-names where
$|L|\leq\omega_1$. Thus there exists a one to one map of $L$ onto a
subset of $\omega_1$. Using this map to construct an automorphism of
$\Fn(\omega_2,2)$ and replacing $\namef A$ with the appropriate
`image' one can assume that $\namef A$ is an $\Fn(\omega_1,2)$-name
and
\begin{enumerate}[conditions]
\item\label{A.dense} $\force_{\Fn(\omega_1,2)}``\namef A\cap U\not=\varnothing\hbox{
for every }U\in{\cal U}\hbox{ such that }0\in U''$

\item\label{A.free} for every $K\in\K$ there exists an $\Fn(\omega_1,2)$-name
$\namef U_K$ such that $\force_{\Fn(\omega_1,2)}``0\in\namef
U_K\in{\cal U}\hbox{ and }\namef A\cap(\namef
U_K\cap K)=\varnothing''$

\end{enumerate}

Let $A_\alpha=F(\namef A)\cap\alpha$. Let $\Omega_A$ be the set of all
$\alpha\in\omega_1$ such that
\begin{enumerate}[conditions]
\item\label{Aa.dense} $\force_{\Fn(\omega_1,2)}``A_\alpha\cap
U\not=\varnothing\hbox{ for every }U\in\Ua\hbox{ such that }0\in U''$

\item\label{Aa.free} for every $K\in\Ka$ there is a $\namef U_K$
such that $\force_{\Fn(\omega_1,2)}``0\in\namef U_K\in\Ua\hbox{ and }A\cap (K\cap
\namef U_K)=\varnothing''$

\item\label{Aa.indep} the group $G_\alpha$ generated by
$\pi_1(A_\alpha)$ is such that $G_\alpha\cap H_\alpha=\{0\}$ and for
every $g\in G_\alpha$ and any $K\in\Ka$ $(g+H_\alpha)\cap
K\not=\varnothing$ if and only if $g\in K$

\item\label{Aa.closed} $A_\alpha$ is an $\Fn(\alpha,2)$-name
\end{enumerate}

Note that $\Omega_A$ is a closed subset of $\omega_1$. To show that
$\Omega_A$ is unbounded let
$\beta\in\omega_1$ and let $M$ be a countable elementary submodel such
that $\{\namef A, \beta, {\cal U}, \K, \tau, \tau', \Fn(\omega_1,2)\}\subseteq M$ along with the
necessary details of the construction above. Put
$\alpha=M\cap\omega_1$. If $U\in\Ua$ then $U\in{\cal U}_{\alpha'}$ for some
$\alpha'<\alpha$ so by elementarity, $U\in M$. If
$\not\force_{\Fn(\omega_1,2)}``A_\alpha\cap U\not=\varnothing''$ there
exists a $p\in\Fn(\omega_1,2)$ such that
$p\force_{\Fn(\omega_1,2)}``A_\alpha\cap U=\varnothing''$. Let
$p'=p\cap M$. Using~\ref{A.dense} and the elementarity of $M$ find a
$p''\in M$ and a $g\in G\cap M$ such that $p''\leq p'$ and
$p''\force_{\Fn(\omega_1,2)}``g\in\namef A\cap U''$. Since $g\in M$
this implies $p''\force_{\Fn(\omega_1,2)}``g\in A_\alpha\cap U''$
which contradicts the compatibility of $p''$ and
$p$. Thus~\ref{Aa.dense} holds.

Property~\ref{Aa.free} can be shown by first observing that each
$K\in\Ka$ is in $M$ and using~\ref{A.free}.

To show~\ref{Aa.indep} let $g\in G_\alpha$ and $K\in\Ka$ be such that
$g\not\in K$. Then $K\in M$ and $g\in M$ as a finite
sum of elements of $G\cap M$. Let $g\in K'$ for some
$K'\in\Ka$. Since $\tau'$ is a Hausdorff group topology in which $K'$
is a compact subset there exists a $U\in{\cal U}_\delta\subseteq{\cal
U}$ for some $\delta<\omega_1$ such that $0\in U$
and $g\not\in U$. By elementarity, we can assume
that $K',U\in M$ so $\delta<\alpha$ and $U\in\Ua$. Similarly, there
exists a $U'\in{\cal U}$ such that $0\in U'$ and $(g+U')\cap K=\varnothing$ and
$U'\in M$ so $U'\in\Ua$. Thus $H_\alpha\subseteq U'$.

The elementarity of $M$ implies~\ref{Aa.closed}, as $A_{\alpha'}\in M$
is countable for every $\alpha'<\alpha$.

Since $\Omega_A$ is a club there exists an $\alpha\in\omega_1$ such
that $A_\alpha=C_\alpha$. Now~\ref{Aa.dense}--\ref{Aa.indep}
imply \ref{subCa.dense}--\ref{subCa.indep} so~\ref{subCa} implies that
$\force_{\Fn(\omega_1,2)}``\namef A\cap S_{\alpha+1}\hbox{ is infinite
}''$ contradicting~\ref{A.free}.
\end{proof}

A name $\namef A$ with the properties of Lemma~\ref{groupsq} would exist if
$\tau$ and $\tau'$ were different in $V[\G]$.
Thus $\tau$ is a sequential group topology on $G$. The sequential
order $\so(G)\geq2$ since $B\subseteq 
G$ is a compact subset of $G$ such that $\so(B)=2$.
If $\so(G)>2$ in $V[\G]$ there exists a set $A\subseteq G$ such that
$0\not\in[A]_2$ while $0\in\overline{A}$. Since $\so(K)\leq2$ for
every $K\in\K$ by the remark after Example~\ref{psi}, $0\not\in\overline{A\cap K}$. An
$\Fn(\omega_2,2)$-name $\namef A$ of $A$ satisfies~\ref{A2.dense'}
and~\ref{A2.free'} in $V$.

The argument above thus shows
\begin{lemma}\label{agoisqo}
Let $V\vdash\hbox{CH}$ and $\G$ be $\Fn(\omega_2,2)$-generic. There
exists a sequential group $G\in V[\G]$ such that $\so(G)=2$.
\end{lemma}

Combining Lemma~\ref{ncsgisq} an Lemma~\ref{agoisqo} we obtain the main
theorem of this paper which answers the author's question in~\cite{S1}.

\begin{theorem}\label{sggap}
It is consistent with the axioms of ZFC that there are
no countable sequential groups of intermediate sequential order while
there is an uncountable such group.
\end{theorem}

It is natural to ask if sequential groups of arbitrary sequential
order $\alpha<\omega_1$ exist in $V[\G]$. According
to \cite{Dp}, the compact spaces in the well known series constructed
by A.~Bashkirov (see~\cite{B}) can be made Cohen-indestructible. The
group algebraically generated by such a compact space is co-countable
since every quotient of countable pseudocharacter will be covered by
countably many scattered metrizable compact subspaces. Thus the
construction in this section proceeds virtually unchanged to result in
a sequential group of any sequential order $\alpha<\omega_1$.

On the other hand, the group constructed in Lemma~\ref{agoisqo} has
the property that $\overline{A}=\cup\set\overline{A\cap K}:K\in\K.$
for any $A\subseteq G$ thus making $\tau$ ``Fr\'echet $\pmod\K$'' (see
\cite{Ka} for a general discussion of ordinal invariants defined in
this manner). So the following question seems natural.

\begin{question}
Does there exist a sequential group $G\in V[\G]$ of sequential order $2$
such that for any family $\K$ of compact subsets of $G$ there exists a
subset $A\subseteq G$ such that
$\overline{A}\not=\cup\set\overline{A\cap K}:K\in\K.$?
\end{question}

Additionally, it would be interesting to know if the intermediate
sequential order ``reflects'' to a group of size $\omega_1$ in
$V[\G]$.

\begin{question}
Let $G\in V[\G]$ be a sequential group of intermediate sequential
order. Does there exists a sequential group $G'\subseteq G$ of
intermediate sequential order of size $\omega_1$?
\end{question}

\end{document}